\documentclass[letterpaper,11pt,reqno]{amsart}

\makeatletter
\usepackage{amssymb}
\usepackage{latexsym}
\usepackage{amsbsy}
\usepackage{amsfonts}
\usepackage{hyperref}
\usepackage{graphicx}
\usepackage{enumerate}
\usepackage{enumitem}
\usepackage{mathrsfs}
\usepackage{url}
\usepackage{color}

\usepackage{tikz}

\def\marginpar#1{\ignorespaces}

\DeclareMathOperator\argmin{argmin}

\newtheorem{theorem}{Theorem}[section]
\newtheorem{lemma}[theorem]{Lemma}
\newtheorem{proposition}[theorem]{Proposition}
\newtheorem{corollary}[theorem]{Corollary}
\newtheorem{definition}[theorem]{Definition}
\newtheorem{ass}[theorem]{Assumption}

\newtheorem{conj}[theorem]{Conjecture}
\numberwithin{equation}{section}
\makeatother
\begin{document}
\title[Exponential ergodicity of GRBM]{Exponential ergodicity and convergence for generalized reflected Brownian motion}

\author[Wenpin Tang]{{Wenpin} Tang}
\address{Department of Mathematics, University of California, Los Angeles. Email: 
} \email{wenpintang@math.ucla.edu}

\date{\today} 

\begin{abstract}
In this paper we provide convergence analysis for a class of Brownian queues in tandem by establishing an exponential drift condition.
A consequence is the uniform exponential ergodicity for these multidimensional diffusions, including the O'Connell-Yor process.
A list of open problems are also presented. 
\end{abstract}
\maketitle
\maketitle

\section{Introduction and main results}
\quad This paper is concerned with the convergence to global equilibrium for a large class of multidimensional Brownian diffusions. 
We consider the following $d$-dimensional stochastic differential equation (SDE):
\begin{equation}
\label{SDE}
dX_t=dB_t+ \left(\mu+\sum_{i=1}^d U'(n_i \cdot X_t)r_i\right) \mbox{d} t \quad \mbox{and} \quad X_0 \in \mathbb{R}^d,
\end{equation}
where
\begin{itemize}[itemsep = 4 pt]
\item
$(B_t;\, t \geq 0)$ is a $d$-dimensional Brownian motion with covariance matrix $\Gamma$;
\item
$\mu :=(\mu_1, \ldots, \mu_d)^T\in \mathbb{R}^d$ is a vector of drifts;
\item
$n_1, \ldots, n_d \in \mathbb{S}^{d-1}$ are unit vectors; 
\item
$r_1, \ldots, r_d \in \mathbb{R}^d$ are vectors of reflection and $R:=(r_i)_{1 \leq i \leq d} $ is the reflection matrix;
\item
 $U: \mathbb{R} \rightarrow \mathbb{R}$ is a smooth potential function such that $U' \geq 0$. 
\end{itemize}

The SDE \eqref{SDE} was previously considered by O'Connell and Ortmann \cite{OCO}, the strong solution to which is called a {\em generalized reflected Brownian motion} parametrized by $(\Gamma, \mu, R, U)$, or simply {\em GRBM}$(\Gamma, \mu, R, U)$.  
In the sequel, let $||\Gamma||:=\sup\{||\Gamma x||; ||x||=1\}$ be the spectral norm of $\Gamma$.

\subsection{Motivations}
The study of GRBMs is motivated from both queueing theory and interacting particle systems. 
The simplest model in queueing networks is the {\em $M/M/1$ queue},
where customers arrive according to a Poisson process with rate $\lambda$, and are served according to exponential times with rate $\theta$. 
It is well known that if $\lambda < \theta$, then the number of customers converges to a geometric random variable with parameter $\lambda/\theta$.
{\em Queues in tandem} are consecutive $M/M/1$ queues. 
Harrison \cite{Har73,Har78} proposed an approximate model of queues in the {\em heavy traffic limit}, as the {\em traffic intensity} $\lambda/\theta$ tends to $1$.
That is,
\begin{equation}
\label{SDE2}
X_t=B_t+\mu t +RL_t  \in [0,\infty)^d,
\end{equation}
where $(B_t; \, t \ge 0)$, $\mu$ and $R$ are defined as above, and $L=(L^i_t;\, t \geq 0)_{1 \leq i \leq d}$ is the local time process satisfying that for all $1 \leq i \leq d$,
\begin{itemize}[itemsep = 4 pt]
\item
$L^i$ is continuous and non-decreasing with $L^i_0=0$,
\item
$L^i$ increases only at times $t$ such that $X^i_t=0$.
\end{itemize}

\quad Call the strong solution to \eqref{SDE2} a {\em semimartingale reflected Brownian motion} parametrized by $(\Gamma, \mu, R)$, or simply {\em SRBM}$(\Gamma, \mu, R)$. 
The limit theorems were proved in \cite{Reiman}.
The SDE \eqref{SDE2} was also called the {\em semimartingale reflected Brownian motion in the orthant} by Harrison and Reiman \cite{HR81b, HR81}.
Let $Q$ be obtained by replacing each entry of $I - R$ by its absolute value.
Dupuis and Ishii \cite{DI91} proved that if $||Q|| < 1$, then \eqref{SDE2} has a unique strong solution.
Taylor and Williams \cite{TW93} showed that \eqref{SDE2} has a weak solution which is unique in law if and only if $R$ is completely-$\mathcal{S}$.
The positive recurrence of SRBMs was explored by Dupuis and Williams \cite{DW94}, Chen \cite{Chen96}, and Bramson, Dai and Harrison \cite{Bra11,Bra10,DH12}.
The exponential ergodicity of SRBMs was studied by Budhiraja and Lee \cite{BL07}, Sarantsev \cite{Sar17}, and Blanchet and Chen \cite{BC16}. 
See also Williams \cite{WilliamsO} for a survey.
Recently, SRBMs appear to be a useful tool in the study of rank-dependent Brownian systems, see \cite{BFK,IPBKF,sara1,Shk,TT18}.

\quad The local time process $L$ is often not easy to study. 
A common technique, called the {\em penalty method}, is to substitute the local time process with some drift term which pushes the process inside the domain. 
The choice $U(y) = - e^{-y}$ in \eqref{SDE} corresponds to the {\em generalized Brownian queue in tandem}, introduced by O'Connell and Yor \cite{OY}.
This extends earlier work on heavy traffic queues in tandem by Glynn and Whitt \cite{GW}, and Harrison and Williams \cite{HW92}.
By introducing a parametric family of potentials $U_{\beta}(y)=-\frac{1} {\beta}e^{-\beta y}$  and letting $\beta \rightarrow \infty$, we get the SRBM as weak limit of GRBMs, following the work of Lions and Sznitman \cite{LS84}. 
In the particle system literature, the O'Connell-Yor process is also called the {\em semi-discrete Brownian polymer}, 
which generalizes the low-density limit of a totally asymmetric simple exclusion process ({\em Brownian TASEP}).
This polymer model was proved to be exactly solvable \cite{OC12}, and belongs to the KPZ universality class \cite{BC14,BCF}.

\quad O'Connell and Ortmann \cite[Corollary $4.11$]{OCO} proved that under the {\em generalized skew-symmetry condition}
\begin{equation}
\label{SS}
r_{ij}+r_{ji}=2 \Gamma_{ij} \quad \mbox{for}~1 \leq i \ne j \leq d,
\end{equation}
and under sufficient regularity for $U$,
GRBM$(\Gamma,\mu,R,U)$ has a product-form stationary distribution
\begin{equation}
\label{densityGRBM}
p(x)= \exp \left\{ 2 \left[\sum_{i=1}^d U(x_i)+(2 \Gamma-R)^{-1} \mu \cdot x \right] \right\},
\end{equation}
provided that $\int_{\mathbb{R}^d} p(x)dx < \infty$.
This result is an analog of Williams \cite{WilliamsP,WilliamsO} regarding SRBMs.
Kang and Ramanan \cite{KR14} characterized stationary distributions of reflected diffusions with state-dependent drifts.
L\'epingle \cite{Lep} considered a two-dimensional GRBM with logarithmic potential.

\quad The formula \eqref{densityGRBM} suggests that the GRBM converge exponentially to its stationary distribution. 
The intuition comes from the Poincar\'{e} inequality, see \cite{BCG,CG17} for connections between functional inequalities and rate of convergence for Markov processes.
However,
\begin{itemize}[itemsep = 4 pt]
\item
the GRBM defined by \eqref{SDE} is not necessarily time-reversible, or symmetric;
\item
the stationary distribution of the GRBM has an explicit form only under the generalized skew-symmetric condition \eqref{SS}.
\end{itemize}
So we cannot apply the Poincar\'{e} inequality directly. 
A natural question is whether the rate of convergence is exponential under general conditions. 
The main tool is stochastic stability theory of Markov processes which we recall in Section \ref{s2}.

\subsection{Main results}
For any signed Borel measure $\mu$ on $\mathbb{R}^d$, we define the total variation norm by
$$||\mu||_{TV}:=\underset{|g| \leq 1}{\sup} \int _{\mathbb{R}^d} g d \mu.$$
\begin{definition}
Assume that a $\mathbb{R}^d$-valued Markov process $(Z_t; \,t \geq 0)$ with transition kernel $P^t$ has a unique stationary distribution $\pi$. If there exist $C, \, \delta>0$ and $W: \mathbb{R}^d \rightarrow [1,\infty)$ such that for all $x \in \mathbb{R}^d$ and $t \geq 0$,
\begin{equation}
||P^t(x,\cdot)-\pi(\cdot)||_{TV} \leq C \, W(x)\exp(-\delta t),
\end{equation}
then $(Z_t; \,t \geq 0)$ is said to be $W$-uniformly ergodic with exponent $\delta$.
\end{definition}

\quad Note that the SDE \eqref{SDE} does not always have a unique strong solution unless we impose additional conditions on the input data $(\Gamma, \mu, R, U)$. 
See \cite[Section V]{RW} for background on solutions to SDEs. 
It is well known that the SDE \eqref{SDE} has a strong solution which is pathwise unique under the following conditions.
\begin{enumerate}[itemsep = 4 pt]
\item
$\Gamma$ is strictly positive definite. That is, there exists $\lambda>0$ such that
\begin{equation}
\label{PD}
\xi^{T}\Gamma\xi \geq \lambda ||\xi||^2 \quad \mbox{for all}~ \xi \in \mathbb{R}^d.
\end{equation}
\item
$U'$ is locally Lipschitz. That is, there exists $K_R>0$ such that
\begin{equation}
|U'(y)-U'(z)| \leq K_R|y-z| \quad \mbox{for all}~ |y| \leq R~\mbox{and}~|z| \leq R.
\end{equation}
\item
There exists $K>0$ such that 
\begin{equation}
\label{addcond1}
x \cdot \left(\mu+\sum_{i=1}^d U'(x_i)r_i\right) \leq K(1+||x||^2)  \quad \mbox{for all}~ x=(x_1, \cdots, x_d)\in \mathbb{R}^d.
\end{equation}
\end{enumerate}
Moreover, the unique strong solution is Feller continuous and strong Markov.
The infinitesimal generator of GRBM$(\Gamma, \mu, R, U)$ is given by
\begin{equation}
\label{IG}
\mathcal{L}:=\frac{1}{2} \sum_{i,j=1}^d \Gamma_{ij} \frac{\partial^2}{\partial x_i \partial x_j}+ \sum_{i=1}^d \left(\mu_i+\sum_{j=1}^d U'(x_j) r_{ij}\right)\frac{\partial}{\partial x_i}.
\end{equation}

\quad By proper scaling, we assume that the diagonal entries of the reflection matrix $R$ are all equal to $1$, i.e. 
$r_{ii}=1$ for $1 \leq i \leq d$. 
To prove the rate of convergence for a GRBM, we make the following assumptions on the input data. 

\begin{ass}
\label{hypo}
\begin{enumerate}[itemsep = 4 pt]
\item
The covariance matrix $\Gamma$ is strictly positive definite.
\item
The reflection matrix $R$ is such that $r_{ii}=1$ for $1 \leq i \leq d$, $r_{i,i-1}=-1$ for $2 \leq i \leq d$ and $r_{ij}=0$ for $|i-j| \geq 2$. That is,
\begin{equation}
\label{trid}
R=\begin{pmatrix} 
1 & 0 & 0& \cdots &0\\ 
-1 & 1 & 0& \cdots &0 \\ 
\vdots & \ddots & \ddots& \ddots &\vdots \\ 
0 & \cdots& -1 &  1 & 0\\
 0 & \cdots & 0 &-1& 1\\ 
\end{pmatrix},
\end{equation}
\item
The potential $U$ satisfies that
\begin{enumerate}[itemsep = 4 pt]
\item
$U' \geq 0$ is continuous on $\mathbb{R}$, and is decreasing on $\mathbb{R}_{-}$,
\item
$U'(y) \rightarrow 0$ as $y \rightarrow \infty$, and $U'(y) \rightarrow \infty$ as $y \rightarrow -\infty$,
\item
For $a>b>0$, $U'(ay)/U'(by) \rightarrow \infty$ as $y \rightarrow -\infty$.
\end{enumerate}
\end{enumerate}
\end{ass}
The form \eqref{trid} of the reflection matrix comes from the application to queues in tandem, and the gap process of the Brownian TASEP. 
In particular, Assumption \ref{hypo} is satisfied with $\Gamma=I$ and $U(x)=-e^{-x}$, corresponding to the O'Connell-Yor process.
See Section \ref{s5} for further discussion.

\quad Under Assumption \ref{hypo}, the infinitesimal generator of GRBM$(\Gamma,\mu,R,U)$ simplifies to
\begin{equation}
\label{IGexp}
\mathcal{L}:=\frac{1}{2} \sum_{i,j=1}^d \Gamma_{ij} \frac{\partial^2}{\partial x_i \partial x_j}+ \sum_{i=1}^d \left(\mu_i+U'(x_i)-U'(x_{i-1})\right)\frac{\partial}{\partial x_i},
\end{equation}
with the convention $x_0=\infty$. 
For a measurable set $C \subset \mathbb{R}^d$, write $1_C$ for the indicator function of $C$.
For $r>0$, let $\mathcal{B}_r:=\{x \in \mathbb{R}^d;\, ||x|| \leq r\}$ be the closed ball of radius $r$ centered at the origin.
A function $V: \mathbb{R}^d \rightarrow [1,\infty)$ is said to be norm-like if it is at least twice continuously differentiable, and $V(x) \rightarrow \infty$ as $||x|| \rightarrow \infty$.

\quad The main result below establishes an {\em exponential drift condition} for GRBM$(\Gamma,\mu,R,U)$.

\begin{theorem}
\label{keylem}
Let $\mathcal{L}$ be defined by \eqref{IGexp}. 
Assume that the input data $(\Gamma,\mu,R,U)$ satisfy Assumption \ref{hypo}, and 
$\mu=(\mu_1, \ldots, \mu_d) <0$. 
Then there exist  $b<\infty$, and a norm-like function $V: \mathbb{R}^d \rightarrow [1,\infty)$ such that for arbitrary small $\epsilon>0$,
\begin{equation}
\label{GDCbis}
\mathcal{L}V \leq -\frac{1}{2d||\Gamma||}\left(\min_{1 \leq i\leq d}|\mu_i|^2-\epsilon\right)V+b1_{\mathcal{B}_r} \quad \mbox{for}~r~\mbox{large enough}.
\end{equation}
\end{theorem}

\quad The proof of Theorem \ref{keylem}, which involves an induction over $d$, will be given in Section \ref{s4}.
The function $V$ is called the {\em Lyapunov function}.
The main difficulty is to prove the estimate \eqref{GDCbis} for $x:=(x_1,\ldots, x_d) < 0$,
 in which case 
$\mu_i + U'(x_i) - U'(x_{i-1})$ can either be large positive or large negative since $U'(y) \rightarrow \infty$ as $y \rightarrow - \infty$.
Note that this is the case for the O'Connell-Yor process.
The key to analysis is to estimate $\beta_d(x)$ defined by \eqref{betad}, and
the special form \eqref{trid} of $R$ and the monotonicity of $U'$ lead to manageable estimates via telescopic sums \eqref{201}--\eqref{18}.

\quad A consequence of Theorem \ref{keylem} is the uniform exponential ergodicity for GRBM$(\Gamma,\mu,R,U)$, which we prove in Section \ref{s2}.
As a byproduct, we obtain a tail estimate of the stationary distribution of GRBM$(\Gamma,\mu,R,U)$.
\begin{corollary}
\label{mainbis}
Under the assumptions of Theorem \ref{keylem}, GRBM$(R, \mu, \Gamma, U)$ defined as the strong solution to \eqref{SDE}, has a unique stationary distribution $\pi$, and is uniformly exponentially ergodic. 
That is, let $P^t$ be the transition kernel of GRBM$(R,\mu,\Gamma,U)$, then there exist 
$$W: \mathbb{R}^d \rightarrow [1,\infty) \quad \mbox{and} \quad C(d) > 0,$$
such that
\begin{equation}
||P^t(x, \cdot)-\pi(\cdot)||_{TV} \leq W(x) \exp(-C(d) \, t).
\end{equation}
Moreover, the stationary distribution $\pi$ satisfies the tail estimate
\begin{equation}
\int_{\mathbb{R}^d} V(x) d\pi(x) < \infty.
\end{equation}
\end{corollary}

\quad The exponential ergodicity of GRBM$(\Gamma, \mu, R, U)$ holds for $\mu < 0$, which is different from 
the exponential ergodic condition $R^{-1}\mu < 0$ for SRBM$(\Gamma, \mu, R)$
due to the presence of the potential $U$.
As explained in Section \ref{s2}, the dependence of $C(d)$ is notoriously difficult.
There is no simple way to get the exact rate $C(d)$ from the Lyapunov estimate \eqref{GDCbis}.
Blanchet and Chen \cite{BC16} gave an estimate of $C(d)$ for SRBMs.
Their result relies on a coupling at the minimal element $\{0\}$ in the positive orthant, which is not available for GRBMs.
Eberle, Guillin and Zimmer \cite{EGZ18} provided an estimate of $C(d)$ for general diffusions.
There a one-sided Lipschitz condition for the drift term is required, but this is not necessarily satisfied in our scenario.
In view of \eqref{GDCbis}, we make the following conjecture.
\begin{conj}
\label{cj6}
Under the assumptions of Theorem \ref{keylem}, GRBM$(R, \mu, \Gamma, U)$ is uniformly exponentially ergodic with exponent of order $1/d$. That is, there exist 
$$W: \mathbb{R}^d \rightarrow [1,\infty) \quad \mbox{and} \quad C > 0$$
such that
\begin{equation}
||P^t(x,\cdot)-\pi(\cdot)||_{TV} \leq W(x)\exp\left(-\frac{Ct}{d}\right).
\end{equation}
\end{conj}
Conjecture \ref{cj6} would imply that the relaxation time to global equilibrium is of order $d$ for a class of GRBMs.

\bigskip

{\bf Outline of the paper}: The rest of the paper is organized as follows.
In Section \ref{s2}, we provide background on stochastic stability theory, and prove Corollary \ref{mainbis}.
In Section \ref{s4}, we give a proof of Theorem \ref{keylem}.
In Section \ref{s5}, we apply these results to the O'Connell-Yor process. There several open questions are raised.
\section{Stochastic stability $\&$ exponential ergodicity for GRBM}
\label{s2}
\subsection{Stochastic stability theory}
In this subsection we present the main tool to prove Corollary \ref{mainbis}: 
stochastic stability theory for continuous-time Markov processes developed by Meyn and Tweedie \cite{MT2,MT3}. 
See \cite{MTbook, MT2} for background.

\quad Meyn and Tweedie \cite[Theorem $4.2$]{MT3} provided criteria for a Markov process to be positive Harris recurrent in terms of its infinitesimal generator. 
 \begin{theorem}
 \label{MTADR}
 Let $(Z_t; \, t \geq 0)$ be a $\mathbb{R}^d$-valued Markov process with generator $\mathcal{L}$.
 \begin{enumerate}
 \item  \label{MTADT1}
 If there exist $k>0$, $b<\infty$, a petite set $C \subset \mathbb{R}^d$, and a norm-like function $V: \mathbb{R}^d \rightarrow [1,\infty)$ such that
 \begin{equation*}
 \mathcal{L}V \leq -k+b1_{C},
 \end{equation*}
 then $(Z_t; \, t \geq 0)$ is positive Harris recurrent.
  \item  \label{MTADT2}
If $(Z_t; \,t \geq 0)$ is positive Harris recurrent, then it has a unique stationary distribution.
 \end{enumerate}
 \end{theorem}
 
\quad Meyn and Tweedie \cite[Theorem 6.1]{MT3} also gave a criterion for a Markov process to be uniformly exponentially ergodic. 
\begin{theorem}
\label{dadada}
Let $(Z_t; \, t \geq 0)$ be a $\mathbb{R}^d$-valued Markov process with generator $\mathcal{L}$. 
If all compact sets are petite, and there exist $k>0$, $b<\infty$, a petite set $C \subset \mathbb{R}^d$, and a norm-like function $V: \mathbb{R}^d \rightarrow [1,\infty)$ such that
 \begin{equation*}
 \mathcal{L}V \leq -k V+b1_{C},
 \end{equation*}
 then $(Z_t; \, t \ge 0)$ is uniform exponential ergodic.
\end{theorem}
 
\quad It is well known that under the {\em geometric drift condition}, a Markov chain converges to its equilibrium with rate $\rho^n$ for some $\rho<1$, see \cite[Chapter $15$]{MTbook}. 
Down, Meyn and Tweedie \cite{DMT} extended this result to the continuous setting. 
Under the exponential drift condition, a Markov process converges exponentially to its stationary distribution with some exponent $\delta>0$. But the explicit value or bounds of $\rho<1$ and $\delta>0$ were unknown.
Efforts have been made to derive bounds of $\rho$ and $\delta$ under extra assumptions that
\begin{itemize}[itemsep = 4 pt]
\item
the Markov chain/process is stochastically ordered, and the state space has a minimal element, see \cite{BC16, LMT,LT}.
\item
the Markov chain/process satisfies a minorisation condition: there exists a Borel set $C \subset \mathbb{R}^d$, $t^{*}>0$, $\varepsilon>0$, and a probability distribution $\nu$ on $\mathbb{R}^d$ such that for each Borel set $A \in \mathbb{R}^d$,
$$
P^{t^{*}}(x,A)  \geq \varepsilon \nu(A) \quad \mbox{for all}~x \in C.
$$
But it is difficult to provide good estimates of $(t^{*}, \varepsilon)$.
See \cite{RT,RR,Rminor}.
\end{itemize}
\subsection{Proof of Corollary \ref{mainbis}}
In this subsection we explain how to use the exponential drift condition \eqref{GDCbis} to derive the uniform exponential ergodicity for GRBM$(\Gamma,\mu,R,U)$. 
We begin with two simple lemmas.
\begin{lemma}
\label{prop1}
Under Assumption \ref{hypo}, the SDE \eqref{SDE} has a strong solution which is pathwise unique.
\end{lemma}
\begin{proof}
As we will see in Section \ref{s4}, a key step to prove Theorem \ref{keylem} is the following estimate
\begin{equation}
\label{Khas}
\sum_{i=1}^dx_i(\mu_i+U'(x_i)-U'(x_{i-1})) \leq \left(-\min_{1 \leq i \leq d}|\mu_i|+\epsilon\right) ||x||,
\end{equation}
for arbitrary small $\epsilon>0$ and $||x||$ large enough. Plainly, the Khasminskii non-explosion condition \eqref{addcond1} is satisfied. Combined with the local Lipschitz property of $U'$, we conclude that the SDE \eqref{SDE} has a strong solution which is pathwise unique.
\end{proof}

\begin{lemma}
\label{prop2}
Under Assumption \ref{hypo}, GRBM$(\Gamma,\mu,R,U)$ defined as the strong solution to the SDE \eqref{SDE} is positive Harris recurrent and has a unique stationary distribution.
\end{lemma}
\begin{proof} According to the exponential drift condition \eqref{GDCbis}, there exist $k>0$, $b<\infty$, $R>0$, and a norm-like function $V:\mathbb{R}^d \rightarrow [1,\infty)$ such that
$$\mathcal{L}V \leq -k+b1_{\mathcal{B}_R},$$
where $\mathcal{L}$ is defined as in \eqref{IGexp}. 
According to \cite[Theorem $2.3$]{ST}, a GRBM is Lebesgue-irreducible $T-$process.
It follows from \cite[Theorem $4.1$(i)]{MT2} that for a Lebesgue-irreducible $T$-process, every compact set is petite. In particular, $\mathcal{B}_R$ as a compact set is petite. 
By Theorem \ref{MTADR}, GRBM$(R,\mu,\Gamma,U)$ is positive Harris recurrent and has a unique stationary distribution.
 \end{proof}

\quad The existence and uniqueness of the stationary distribution of GRBM can also be derived from the exponential drift condition \eqref{GDCbis} in a purely analytical way. By \cite[Corollary $1.3$]{BR}, there exists a stationary distribution which is absolute continuous relative to Lebesgue measure with density $p \in L^{d/(d-1)}(\mathbb{R}^d)$. The uniqueness follows from  \cite[Example $5.1$]{BRS}. 

\begin{proof}[Proof of Corollary \ref{mainbis}]
Lemma \ref{prop1} and \ref{prop2} guarantees that GRBM$(R, \mu, \Gamma, U)$ is well-defined, and has a unique stationary distribution.
It suffices to apply Theorem \ref{dadada} with Theorem \ref{keylem} to conclude.
\end{proof}
\section{Exponential drift condition for GRBM}
\label{s4}
\quad In this section we prove Theorem \ref{keylem} by induction on $d$ -- the dimension of GRBM.
To proceed further, we need the following notations.
Let $\epsilon>0$ chosen to be small enough and $L>0$ chosen to be large enough. Define
\begin{itemize}[itemsep = 4 pt]
\item 
$\mu^{(d)}_{\min}(\epsilon):=-\underset{1 \leq i \leq d}{\min}|\mu_i|+\epsilon<0$,
\item
$r_{+}(d,\epsilon)>0$ such that for all $1 \le i \le n$, 
$\mu_i+U'(x) \leq \mu^{(d)}_{\min}(\epsilon)$ for $x \geq r_{+}(d,\epsilon)$,
\item
$r_{-}(d,\epsilon,L)>0$ such that for all $1 \le i \le n$, $\mu_i+U'(x) \geq |\mu^{(d)}_{\min}(\epsilon)|+L$ for $x \leq -r_{-}(d,\epsilon,L)$.
\end{itemize}
To avoid heavy notations, we abandon the dependance on $(d,\epsilon,L)$, and write $\mu_{\min}^{(d)}$, $r_{+}$, $r_{-}$ instead of $\mu^{(d)}_{\min}(\epsilon)$, $r_{+}(d,\epsilon)$, $r_{-}(d,\epsilon,L)$.

\begin{proof}[Proof of Theorem \ref{keylem}]
Define $V: \mathbb{R}^d \rightarrow [1, \infty)$ by
\begin{equation}
\label{V}
V(x):=\exp \left(\lambda \phi (||x||)\right) \quad \mbox{for}~\lambda>0,
\end{equation}
where $\phi: \mathbb{R}_{+} \rightarrow \mathbb{R}_{+}$ is an increasing $\mathcal{C}^{2}$ function such that $\phi(s)=0$ for $s \leq \frac{1}{2}$, and $\phi(s)=s$ for $s \geq 1$. Let $\psi(x):=\phi(||x||)$. We get
$$DV(x)= \lambda D\psi(x) V(x),$$
$$D^{2}V(x)= \left(\lambda D^{2}\psi(x) + \lambda^2 D\psi(x)D\psi(x)^{T} \right)V(x).$$
Note that for $||x|| \geq 1$, $||D\psi(x)||=1$ and
$$||D^2\psi(x)||=\left| \left|\frac{I_d}{||x||}-\frac{xx^T}{||x||^3}\right|\right| \leq \frac{2}{||x||} \rightarrow 0 \quad \mbox{as}~||x|| \rightarrow \infty.                                             
$$
So there exists $r_{\epsilon} \geq 1$ such that for $x \notin \mathcal{B}_{r_{\epsilon}}$, we have $||D^2\psi(x)|| \leq \epsilon$. 
In this case, we have
\begin{equation}
\label{1}
\left|\sum_{i,j=1}^{d} \Gamma_{ij} \frac{\partial^2 V}{\partial x_i \partial x_j}(x)\right| \leq d ||\Gamma|| ||D^2V(x)|| \leq d ||\Gamma|| (\lambda \epsilon+ \lambda^2) V(x),
\end{equation}
and
\begin{equation*}
\sum_{i=1}^d \left(\mu_i+U'(x_i)-U'(x_{i-1}) \right)\frac{\partial \psi}{\partial x_i}(x) =\sum_{i=1}^d \left(\mu_i+ U'(x_i)-U'(x_{i-1}) \right)\frac{x_i}{||x||}.                                                  
\end{equation*}
Consequently, for $||x||$ large enough,
\begin{equation}
\label{1000}
\mathcal{L}V(x) \leq \left( \frac{1}{2}d||\Gamma||\lambda^2+(\beta_d(x) + \mathcal{O}(\epsilon))\lambda \right)V(x),
\end{equation}
where 
\begin{equation}
\label{betad}
\beta_d(x):=(\mu_1+U'(x_1))\frac{x_1}{||x||}+\sum_{k=2}^d (\mu_k+U'(x_k)-U'(x_{k-1}))\frac{x_k}{||x||}.
\end{equation}

The case $d=1$ is straightforward.
We consider $d = 2$ as the base case.

\bigskip\noindent
\textbf{Step 1 ($d=2$):} By \eqref{1000},
\begin{equation}
\label{100}
\mathcal{L}V(x) \leq \left[ ||\Gamma|| \lambda^2 +(\beta_2(x)+ \mathcal{O}(\epsilon) )\lambda \right]V(x),
\end{equation}
where 
\begin{equation*}
\beta_2(x):=(\mu_1+U'(x_1))\frac{x_1}{||x||}+(\mu_2+U'(x_2)-U'(x_1))\frac{x_2}{||x||}.
\end{equation*}
\begin{lemma}
\label{2d}
$\beta_2(x) \leq \mu_{\min}^{(2)}+ \mathcal{O}(\epsilon)$ for $||x|| \geq \max \left(r_{+}/\epsilon, r_{-}/\epsilon, r'_{\epsilon}\right)$, where $r'_{\epsilon}$ is given in the proof.
\end{lemma}
\begin{proof} 
There are four cases according to the signs of $(x_1,x_2)$.

\textbf{Case 1:} $x_1 \geq 0$ and $x_2 \geq 0$. 
\begin{enumerate}
\item
If $\frac{x_1}{||x||} \geq \epsilon$ and $\frac{x_2}{||x||} \geq \epsilon$, then $x_1 \geq r_{+}$ and $x_2 \geq r_{+}$. We have 
$$\mu_1+U'(x_1) \leq \mu_{\min}^{(2)} \quad \mbox{and} \quad \mu_2+U'(x_2)-U'(x_1) \leq \mu_2+U'(x_2) \leq \mu_{\min}^{(2)}.$$
Therefore, 
$$\beta_2(x) \leq \mu_{\min}^{(2)}\left(\frac{x_1}{||x||}+\frac{x_2}{||x||}\right) \leq \mu_{\min}^{(2)}.$$
\item
If $\frac{x_1}{||x||} \leq \epsilon$, then $\frac{x_2}{||x||} \geq \sqrt{1-\epsilon^2}$. 
We get $x_2 \geq \sqrt{1-\epsilon^2} r_{+} / \epsilon \geq r_{+}$, so
$$\mu_2+U'(x_2)-U'(x_1) \leq \mu_{\min}^{(2)}.$$
Moreover, $\mu_1+U'(x_1) =\mathcal{O}(1)$ since $x_1 \geq 0$. 
Thus,
$$\beta_2(x) \leq \mathcal{O}(\epsilon)+\mu_{\min}^{(2)} \frac{x_2}{||x||} \leq \mu_{\min}^{(2)}+\mathcal{O}(\epsilon).$$
\item
If $\frac{x_2}{||x||} \leq \epsilon$, then $\frac{x_1}{||x||} \geq \sqrt{1-\epsilon^2}$. 
By symmetry, we get
$$\beta_2(x) \leq \mu_{\min}^{(2)} \frac{x_1}{||x||}+\mathcal{O}(\epsilon) \leq \mu_{\min}^{(2)}+\mathcal{O}(\epsilon).$$
\end{enumerate}

\textbf{Case 2:} $x_1 \leq 0$ and $x_2 \geq 0$.
\begin{enumerate}
\item
If $\frac{x_1}{||x||} \leq -\epsilon$ and $\frac{x_2}{||x||} \geq \epsilon$, then $x_1 \leq -r_{-}$ and $x_2 \geq r_{+}$. We have 
$$\mu_1+U'(x_1) \geq |\mu_{\min}^{(2)}| \quad \mbox{and} \quad \mu_2+U'(x_2)-U'(x_1) \leq \mu_2+U'(x_2) \leq \mu_{\min}^{(2)}.$$
Thus, $$\beta_2(x) \leq \mu_{\min}^{(2)}\left(\frac{-x_1}{||x||}+\frac{x_2}{||x||}\right) \leq \mu_{\min}^{(2)}.$$
\item
If $-\epsilon \leq \frac{x_1}{||x||} \leq 0$, then $\frac{x_2}{||x||} \geq \sqrt{1-\epsilon^2}$. We get $x_2 \geq \sqrt{1-\epsilon^2}r_{+}/\epsilon \geq r_{+}$, so
$$\mu_2+U'(x_2)-U'(x_1) \leq \mu_{\min}^{(2)}.$$
Since $\mu_1+U'(x_1) \geq \mu_1$, we have
$$\beta_2(x) \leq -\mu_1 \epsilon+\mu_{\min}^{(2)} \sqrt{1-\epsilon^2} = \mu_{\min}^{(2)}+\mathcal{O}(\epsilon).$$
\item
If $0 \leq \frac{x_2}{||x||} \leq \epsilon$, thus $\frac{x_1}{||x||} \leq -\sqrt{1-\epsilon^2}$. 
We have $x_1 \leq -\sqrt{1-\epsilon^2}r_{-}/\epsilon \leq -r_{-}$, so
$$\mu_1+U'(x_1) \geq |\mu_{\min}^{(2)}|.$$
For $L$ large enough, $\mu_2+U'(x_2)-U'(x_1) \leq \mu_2+ \sup_{x >0} U'(x)-L \leq 0$. 
Therefore,
$$\beta_2(x) \leq |\mu_{\min}^{(2)}|\frac{x_1}{||x||} \leq \mu_{\min}^{(2)} \sqrt{1-\epsilon^2}=\mu_{\min}^{(2)}+\mathcal{O}(\epsilon).$$
\end{enumerate}

\textbf{Case 3:} $x_1 \geq 0$ and $x_2 \leq 0$.
\begin{enumerate}
\item
If $\frac{x_1}{||x||} \geq \epsilon$ and $\frac{x_2}{||x||} \leq -\epsilon$, then $x_1 \geq r_{+}$ and $x_2 \leq -r_{-}$. We have 
$$\mu_1+U'(x_1) \geq |\mu_{\min}^{(2)}|,$$
and
$$\mu_2+U'(x_2)-U'(x_1) \geq |\mu_{\min}^{(2)}|~\mbox{for}~L~\mbox{large enough}.$$
Thus, $$\beta_2(x) \leq \mu_{\min}^{(2)}\left(\frac{x_1}{||x||}+\frac{-x_2}{||x||}\right) \leq \mu_{\min}^{(2)}.$$
\item
If $\frac{x_1}{||x||} \leq \epsilon$, then $\frac{x_2}{||x||} \leq -\sqrt{1-\epsilon^2}$. 
We get $x_2 \leq -\sqrt{1-\epsilon^2}r_{-}/\epsilon \leq -r_{-}$, so
$$\mu_2+U'(x_2)-U'(x_1) \geq |\mu_{\min}^{(2)}|~\mbox{for}~L~\mbox{large enough}.$$
Moreover, $\mu_1+U'(x_1) = \mathcal{O}(1)$ since $x_1 \geq 0$. Thus,
$$\beta_2(x) \leq \mathcal{O}(\epsilon)+\mu_{\min}^{(2)} \sqrt{1-\epsilon^2} = \mu_{\min}^{(2)}+\mathcal{O}(\epsilon).$$
\item
If $-\epsilon \leq \frac{x_2}{||x||} \leq 0$, thus $\frac{x_1}{||x||} \geq \sqrt{1-\epsilon^2}$. 
We have $x_1 \geq \sqrt{1-\epsilon^2}r_{+}/\epsilon \geq r_{+}$, so
$$\mu_1+U'(x_1) \leq \mu_{\min}^{(2)}.$$
Since $\mu_2+U'(x_2)-U'(x_1) \geq \mu_2 - \sup_{x >0} U'(x)$, we get
$$\beta_2(x) \leq \mu_{\min}^{(2)} \sqrt{1-\epsilon^2}-\left(\mu_2 - \sup_{x >0} U'(x)\right) \epsilon=\mu_{\min}^{(2)}+\mathcal{O}(\epsilon).$$
\end{enumerate}

\textbf{Case 4:} $x_1 \leq 0$ and $x_2 \leq 0$.
\begin{enumerate}
\item
If $\frac{x_1}{||x||} \leq -\epsilon$ and $\frac{x_2}{||x||} \leq -\epsilon$, we have
\begin{align}
\beta_2(x) &=\mu_1 \frac{x_1}{||x||} +\frac{(x_1-x_2)U'(x_1)}{||x||}+\frac{x_2U'(x_2)}{||x||}+\mu_2 \frac{x_2}{||x||} \notag\\
                &\leq \mathcal{O}(1) +\frac{(x_1-x_2)U'(x_2)+x_2U'(x_2)}{||x||} \label{201} \\
                &=\mathcal{O}(1)+\frac{x_1U'(x_2)}{||x||}  \notag\\
                & \leq \mathcal{O}(1)-\epsilon U'(-\epsilon ||x||) \rightarrow -\infty \quad \mbox{as}~||x|| \rightarrow \infty, \notag
\end{align}
where the inequality \eqref{201} follows from the fact that $\mu_1 \frac{x_1}{||x||}$, $\mu_2 \frac{x_2}{||x||} = \mathcal{O}(1)$ and $(x_1-x_2)(U'(x_1)-U'(x_2)) \leq 0$.
\item
If $\frac{x_1}{||x||} \geq -\epsilon$, then $\frac{x_2}{||x||} \leq -\sqrt{1-\epsilon^2}$. We have
\begin{align}
\beta_2(x) &=(\mu_1+U'(x_1))\frac{x_1}{||x||}+(\mu_2+U'(x_2)-U'(x_1))\frac{x_2}{||x||} \notag\\
               &\leq -\mu_1 \epsilon - (\mu_2+U'(-\sqrt{1-\epsilon^2}||x||)-U'(-\epsilon ||x||)) \sqrt{1-\epsilon^2}  \notag\\
               &\rightarrow -\infty \quad \mbox{as}~||x|| \rightarrow \infty.\notag
\end{align}
\item
If $\frac{x_2}{||x||} \geq -\epsilon$, then $\frac{x_1}{||x||} \leq -\sqrt{1-\epsilon^2}$. We have
\begin{align}
\beta_2(x) &=\mu_1 \frac{x_1}{||x||} + \frac{(x_1-x_2)U'(x_1)}{||x||}+(\mu_2+U'(x_2))\frac{x_2}{||x||} \notag\\
               & \leq \mathcal{O}(1)+(\epsilon-\sqrt{1-\epsilon^2})U'(-\sqrt{1-\epsilon^2}||x||)-\mu_2 \epsilon  \notag\\
                & \rightarrow -\infty \quad \mbox{as}~||x|| \rightarrow \infty.\notag
\end{align}
\end{enumerate}
It suffices to take $r'_{\epsilon}>0$ such that $\beta_2(x) \leq \mu_{\min}^{(2)}$ for $x \leq 0$ and $||x|| \geq r'_{\epsilon}$.  
\end{proof}

By Lemma \ref{2d} and \eqref{100}, we get for $||x||$ large enough,
$$\mathcal{L}V(x) \leq \left( ||\Gamma|| \lambda^2+\mu_{\min}^{(2)} \lambda \right)V(x).$$
By taking $\lambda=-\frac{\mu_{\min}^{(2)}}{2||\Gamma||}$, we have
$$\mathcal{L}V(x) \leq \left( -\frac{1}{4||\Gamma||} (\mu_{\min}^{(2)})^2 \right)V(x).$$

\bigskip\noindent
\textbf{Step 2 ($d-1 \rightarrow d$):} We prove the following lemma by induction on $d$. The case $d = 2$ was proved in Lemma \ref{2d}.
\begin{lemma}
\label{claim}
$\beta_d(x) \leq \mu_{\min}^{(d)}$ for $||x||$ large enough, and $\beta_d(x) \rightarrow -\infty$ as $x<0$ and $||x|| \rightarrow \infty$.
\end{lemma}
\begin{proof} 
Let $i^{+}:=\sup\{i \geq 0;\, x_i \geq 0\}$, with the convention $i^{+}=0$ if $x<0$. 
There are three cases.

{\bf Case 1:} If $i^{+} \geq 2$, then we have
\begin{align}
\beta_d(x) &=\left[(\mu_1+U'(x_1))\frac{x_1}{||x||}+\sum_{k=2}^{i^{+}-1}(\mu_k+U'(x_k)-U'(x_{k-1}))\frac{x_k}{||x||} \right]-U'(x_{i^{+}-1})\frac{x_{i^{+}}}{||x||} \notag\\
                   &\quad +\left[(\mu_{i^{+}}+U'(x_{i^{+}}))\frac{x_{i^{+}}}{||x||}+\sum_{k=i^{+}+1}^{d}(\mu_k+U'(x_k)-U'(x_{k-1}))\frac{x_k}{||x||} \right]  \notag\\
                   & \leq \beta_{i^{+}-1}(x_1, \ldots, x_{i^{+}-1})\frac{\sqrt{\sum_{k=1}^{i^{+}-1}x_k^2}}{||x||}+\beta_{d-i^{+}+1}(x_{i^{+}}, \ldots, x_d)\frac{\sqrt{\sum_{k=i^{+}}^{d}x_k^2}}{||x||}.  \notag                 
\end{align}
By induction hypothesis, $\beta_{i^{+}-1}(x_1, \ldots, x_{i^{+}-1}) \le \mu_{\min}^{(i^{+}-1)}$ for $||(x_1, \ldots, x_{i^{+}-1})|| > r_{d-1}$
and
$\beta_{d-i^{+}+1}(x_{i^{+}}, \ldots, x_{d}) \le \mu_{\min}^{(d -i^{+}+1)}$ for $||(x_{i^{+}}, \ldots, x_{d})|| > r_{d-1}$.
There are three subcases. 
If $||(x_1, \ldots, x_{i^{+}-1})|| \le r_{d-1}$, then 
$$\frac{\sqrt{\sum_{k=1}^{i^{+}-1}x_k^2}}{||x||} \rightarrow 0 \quad \mbox{and} \quad \frac{\sqrt{\sum_{k=i^{+}}^{d}x_k^2}}{||x||} \rightarrow 1 \quad \mbox{as } ||x|| \rightarrow \infty.$$
So we get
$$\beta_d(x) \leq o(1) + \mu_{\min}^{(d -i^{+}+1)} \le \mu_{\min}^{(d)}.$$
The same result holds if $||(x_{i^{+}}, \ldots, x_{d})|| \le r_{d-1}$. 
Assume that $||(x_1, \ldots, x_{i^{+}-1})|| > r_{d-1}$ and $||(x_{i^{+}}, \ldots, x_{d})|| > r_{d-1}$. We have
$$\beta_d(x) \leq \mu_{\min}^{(i^{+}-1)} \frac{\sqrt{\sum_{k=1}^{i^{+}-1}x_k^2}}{||x||}+\mu_{\min}^{(d -i^{+}+1)}\frac{\sqrt{\sum_{k=i^{+}}^{d}x_k^2}}{||x||} \leq \mu_{\min}^{(d)}.$$

{\bf Case 2:} If $i^{+}=1$, then we have
$$\beta_d(x) =\mu_1 \frac{x_1}{||x||}+\frac{(x_1-x_2)U'(x_1)}{||x||} +\beta_{d-1}(x_2, \ldots, x_d)\frac{\sqrt{||x||^2-x_1^2}}{||x||}.$$
\begin{enumerate}
\item
If $\frac{x_1}{||x||} \leq \epsilon$, then $\frac{\sqrt{||x||^2-x_1^2}}{||x||} \geq \sqrt{1-\epsilon^2}$.
By induction hypothesis,
$\beta_{d-1}(x_2, \ldots, x_d) \rightarrow -\infty$ as $||x|| \rightarrow \infty$. 
Therefore,
$$\beta_d(x) \leq \mathcal{O}(1)+\sqrt{1-\epsilon^2}\beta_{d-1}(x_2, \ldots, x_d) \rightarrow -\infty \quad \mbox{as}~||x|| \rightarrow \infty.$$
\item
If $\frac{x_1}{||x||} \geq \epsilon$, then $U'(x_1) \leq \epsilon$ for $||x||$ large enough. 
As a consequence,
$$\beta_d(x) \leq \mu_1 \frac{x_1}{||x||}+ \epsilon +\beta_{d-1}(x_2, \ldots, x_d) \frac{\sqrt{||x||^2-x_1^2}}{||x||}.$$
Similar as Case $1$, we get
$\beta_d(x) \leq  \mu_{\min}^{(d)}$.
\end{enumerate}

{\bf Case 3:} If $i^{+}=0$, i.e. $x <0$, then we write
$\beta_d(x)=\sum_{k=1}^d \mu_k \frac{x_k}{||x||}+\gamma_d(x),$
where
$$\gamma_d(x):=\sum_{k=1}^{d-1}\frac{(x_k-x_{k+1})U'(x_k)}{||x||}+\frac{x_dU'(x_d)}{||x||}.$$
It suffices to prove that $\gamma_d(x) \rightarrow -\infty$ as $||x|| \rightarrow \infty$. 
Let $i_1:=\underset{1 \leq k \leq d}{\argmin}~x_k$.

If $i_1 \geq 2$, then we have
\begin{align}
\gamma_d(x) & = \sum_{k=1}^{i_1-1}\frac{(x_k-x_{k+1})U'(x_k)}{||x||}+\frac{(x_{i_1}-x_{i_1+1})U'(x_{i_1})}{||x||} \notag\\
                          &\qquad \qquad \quad \quad \quad \quad+\sum_{k=i_1+1}^{d-1}\frac{(x_k-x_{k+1})U'(x_k)}{||x||}+\frac{x_dU'(x_d)}{||x||} \notag\\
& \leq \sum_{k=1}^{i_1-1}\frac{(x_k-x_{k+1})U'(x_k)}{||x||} +\frac{(x_{i_1}-x_{i_1+1})U'(x_{i_1-1})}{||x||} \notag\\
&\qquad \qquad \quad \quad \quad \quad +\sum_{k=i_1+1}^{d-1}\frac{(x_k-x_{k+1})U'(x_k)}{||x||}+\frac{x_dU'(x_d)}{||x||}  \label{18} \\
&=\sum_{k=1}^{i_1-2}\frac{(x_k-x_{k+1})U'(x_k)}{||x||} +\frac{(x_{i_1-1}-x_{i_1+1})U'(x_{i_1-1})}{||x||} \notag\\
&\qquad \qquad \quad \quad \quad \quad +\sum_{k=i_1+1}^{d-1}\frac{(x_k-x_{k+1})U'(x_k)}{||x||}+\frac{x_dU'(x_d)}{||x||} \notag\\
&=\gamma_{d-1}(x_1, \ldots, x_{i_1 - 1}, x_{i_1 + 1}, \ldots, x_d) \frac{\sqrt{||x||^2-x_{i_1}^2}}{||x||}, \notag
\end{align}
where the inequality follows from the fact that $x_{i_1}-x_{i_1+1} \leq 0$ and $U'(x_{i_1}) \geq U'(x_{i_1-1})$ by minimality of $x_{i_1}$.
If $\frac{x_{i_1}}{||x||} \geq -(1-\epsilon)$, then
$$\gamma_d(x) \leq \gamma_{d-1}(x_1, \ldots, x_{i_1 - 1}, x_{i_1 + 1}, \ldots, x_d) \sqrt{2 \epsilon -\epsilon^2} \rightarrow -\infty \quad \mbox{as}~||x|| \rightarrow \infty.$$
If $\frac{x_{i_1}}{||x||} \leq -(1-\epsilon)$, then $\frac{x_k}{||x||} \geq -\epsilon$ for all $k \neq i_1$. 
Consequently, 
$$
\gamma_d(x) \leq \mathcal{O}(U'(-\epsilon||x||))+(-1+2\epsilon)U'(-(1-\epsilon)||x||) \rightarrow -\infty \quad \mbox{as }||x|| \rightarrow \infty.
$$

If $i_1=1$, let $i_2:=\argmin_{2 \le k \le d} x_k$. The same argument shows that $\gamma_d(x) \rightarrow -\infty$ as $||x|| \rightarrow \infty$ for $i_2 \geq 3$. 
We continue this algorithm and the only remaining case is $x_1 \leq x_2 \leq \cdots \leq x_d \leq 0$. 
In this case, we get
\begin{align}
\gamma_d(x )&\leq \frac{\sum_{k=1}^{d-1}(x_k-x_{k+1})+x_d}{||x||}U'(x_d) \notag\\
                         &=\frac{x_1}{||x||}U'(x_d) \leq -\frac{1}{\sqrt{d}}U'(x_d) \rightarrow -\infty \quad \mbox{as}~x_d  \rightarrow -\infty. \notag
\end{align}
Now assume that $x_d \geq -r^{*}$ for some $r^{*}>0$. We have
\begin{align*}
\gamma_d(x) &\leq \frac{x_1-x_d}{||x||}U'(x_{d-1})+\frac{x_d}{||x||}U'(x_d)\\
&=\left(-\frac{1}{\sqrt{d}}+o(1)\right)U'(x_{d-1})+\mathcal{O}(1) \rightarrow -\infty \quad \mbox{as}~x_{d-1}  \rightarrow -\infty.
\end{align*}
So it suffices to consider the case $x_{d} \geq x_{d-1} \geq -r^{*}$ for some $r^{*}>0$. We repeat the procedure until $x_d \geq x_{d-1} \geq \cdots \geq x_2 \geq -r^{*}$ for some $r^{*}>0$. 
Then we have 
$$\frac{x_1}{||x||} \leq -\frac{\sqrt{||x||^2-dr^{*2}}}{||x||} \rightarrow -1 \quad \mbox{as}~||x|| \rightarrow \infty.$$
The above condition implies that
$$\gamma_d(x) \leq (-1+o(1)) U'(-(1+o(1))||x||) +\mathcal{O}(1) \rightarrow -\infty \quad \mbox{as}~||x|| \rightarrow \infty.$$
\end{proof}

By Lemma \ref{claim} and \eqref{1000}, we get for $||x||$ large enough,
$$\mathcal{L}V(x) \leq \left(\frac{1}{2}d||\Gamma||\lambda^2+\mu_{\min}^{(d)}\lambda \right)V(x).$$
By taking $\lambda= -\frac{\mu_{\min}^{(d)}}{d||\Gamma||}$, we have
$$\mathcal{L}V(x) \leq \left(-\frac{1}{2d||\Gamma||} (\mu_{\min}^{(d)})^2 \right)V(x).$$
\end{proof}
\section{Brownian diffusions with hard and soft reflection}
\label{s5}
\quad In this section we apply Theorem \ref{keylem} and Corollary \ref{mainbis} to a class of Brownian diffusions with soft reflection, including the O'Connell-Yor process. 
We compare the Brownian TASEP to these diffusions with soft reflection, and present several conjectures regarding the rate of convergence as the dimension $d$ is large. \\\\
{\bf Brownian TASEP}

\quad Consider the Brownian TASEP on the real line. There are $d$ particles with positions $Z^h_1, \cdots,Z^h_d$ such that $Z^h_1(t) \leq \cdots \leq Z^h_d(t)$ for all $t \geq 0$.
The leftmost particle $Z^h_1$ evolves as Brownian motion with drift $\mu_1$.
The second leftmost particle $Z^h_2$ evolves as Brownian motion with drift $\mu_2$ reflected off $Z^h_1$, and so on. 

\quad It is well known that this process is governed by the following SDE:
\begin{equation*}
dZ_1^h(t)=\mu_1 dt+dB_1(t),
\end{equation*}
\begin{equation}
\label{TASEP}
dZ_i^h(t)=\mu_i dt+\frac{1}{\sqrt{2}} (dL_{i-1,i}(t)-dL_{i,i+1}(t))+dB_i(t) \quad  \mbox{for}~2 \leq i \leq d,
\end{equation}
where $B:=(B_i(t); \, t \ge 0)_{1 \le i \le d}$ is a $d$-dimensional Brownian motion with the identity covariance matrix, 
and 
$L_{j,j+1}$ is the local time process of the semimartingale $(Z_{j+1}^h-Z_j^h)/\sqrt{2}$, with the convention $L_{d,d+1}:=0$. 
See Pal and Pitman \cite[Section 2]{PalPitman}.

\quad We consider the gap process $G^h:=(Z_{i+1}^h(t)-Z_i^h(t); \, t \geq 0)_{1 \leq i \leq d-1}$ of the Brownian TASEP.
It was proved in \cite[Section 4]{BFK} that the gap process $(G^h(t); \, t \geq 0)$ is a $(d-1)$-dimensional SRBM$(\Gamma, \tilde{\mu}, R)$, with the reflection matrix 
\begin{equation*}
R:=\begin{pmatrix} 
1 & -1/2 & 0& \cdots &0\\ 
-1/2 & 1 & -1/2& \cdots &0 \\ 
\vdots & \ddots & \ddots& \ddots &\vdots \\ 
0 & \cdots& -1/2 &  1 & -1/2\\
 0 & \cdots & 0 &-1/2& 1\\ 
\end{pmatrix},
\end{equation*}
the drift $\tilde{\mu}:=(\mu_{i+1}-\mu_i)_{1 \leq i \leq d-1}$, and the covariance matrix $\Gamma := 2R$. 
Let $\nu_i:=\sum_{k=1}^i\mu_k-\frac{k}{d}\sum_{k=1}^d\mu_k$ for $1 \leq i \leq d-1$. 
Sarantsev \cite{Sar17} proved that the gap process $G^h$ of the Brownian TASEP has a unique stationary distribution if and only if $\nu_i <0$ for all $1 \leq i \leq d-1$. 
Further by letting $\mathcal{L}^h$ be the infinitesimal generator of $G^h$, there exist a norm-like function $V$ and $b<\infty$ such that
\begin{equation*}
\mathcal{L}^h V \le -K^h V + b 1_{\mathcal{B}_r} \quad \mbox{for}~r~\mbox{large enough},
\end{equation*}
where
\begin{equation}
\label{Khh}
K^h: = \frac{4}{d}\left(1-\cos\frac{\pi}{d}\right)^3\left(1+\cos \frac{\pi}{d}\right)^{-1} \min_{1 \leq i\leq d-1}\nu_i^2.
\end{equation}
Consequently, the gap process $G^h$ is uniformly exponentially ergodic. \\\\
\textbf{Brownian diffusions with soft reflection}

\quad We replace the local time process in \eqref{TASEP} with soft reflection $U'$. 
Precisely, the particle system $Z^s_1, \cdots ,Z^s_d$ is governed by the following SDE:
\begin{equation*}
dZ_1^s(t)=\mu_1 dt+dB_1(t),
\end{equation*}
\begin{equation}
\label{Polymer}
dZ_i^s(t)=\left(\mu_i +U'(Z_{i+1}^s(t)-Z_{i}^s(t))\right)dt+dB_i(t) \quad  \mbox{for}~2 \leq i \leq d,
\end{equation}
where $U$ is a potential function satisfying Assumption \ref{hypo}. 
This multidimensional diffusion is the O'Connell-Yor process with the choice $U(x) = - e^{-x}$.

\quad Consider the gap process $G^s:=(Z_{i+1}^s(t)-Z_i^s(t); \, t \geq 0)_{1 \leq i \leq d-1}$ of the Brownian diffusion with soft reflection. 
We write
\begin{equation*}
dG^s(t)=\left(\tilde{\mu}+ \sum_{i=1}^d U'(G^s_i(t))r_{i}\right)dt+d\tilde{B}(t).
\end{equation*}
So the gap process $(G^s(t); \, t \ge 0)$ is a $(d-1)$-dimensional GRBM$(\Gamma, \tilde{\mu}, R, U)$, 
where the reflection matrix $R$ is given by \eqref{trid}, the drift $\tilde{\mu}$ and the covariance matrix $\Gamma$ are the same as those defined for the Brownian TASEP. 
The following proposition is a consequence of Theorem \ref{keylem}.
\begin{proposition}
\label{ROCS}
Assume that the input data $(R,\tilde{\mu},\Gamma,U)$ satisfy the assumptions in Theorem \ref{keylem}.
Then the gap process $G^s$ of the Brownian diffusion with soft reflection has a unique stationary distribution.
Let $\mathcal{L}^s$ be the infinitesimal generator of $G^s$. Then there exist a norm-like function $V$ and $b < \infty$ such that
\begin{equation*}
\mathcal{L}^sV \le -K^s V + b 1_{\mathcal{B}_r} \quad \mbox{for}~r~\mbox{large enough},
\end{equation*}
where 
\begin{equation}
\label{Kt}
K^s :=\frac{1}{4d (1+\cos \frac{\pi}{d})} \min_{1 \leq i \leq d}|\tilde{\mu}_i|^2.
\end{equation}
Consequently, the gap process $G^s$ is uniformly exponentially ergodic.
\end{proposition}

\quad From the exponential drift conditions \eqref{Khh}--\eqref{Kt}, we see that
\begin{equation}
\label{comparison}
K^h \sim d^{-7} \quad \mbox{and} \quad K^{s} \sim d^{-1} \quad \mbox{as } d \rightarrow \infty.
\end{equation}
This suggests that Brownian diffusions with soft reflection converges faster than those with hard reflection. 
Blanchet and Chen \cite{BC16} proved a bound $d^{-4} (\log d)^{-2}$ for the exponent of the gap process $G^h$. 
The exact rate of convergence of the gap process $G^s$ (resp. $G^h$) remains open.
We leave these problems for future research.

\bigskip

{\bf Acknowledgment:} 
We thank Misha Shkolnikov and Andrey Sarantsev for helpful discussions.
We are also grateful to an Associate Editor and two anonymous referees for their valuable suggestions and various pointers to the literature, which improved the presentation of the paper.

\bibliographystyle{plain}
\bibliography{Rate}
\end{document}